\documentclass[11pt]{article}

\usepackage[a4paper,margin=29mm]{geometry}
\usepackage{amsmath,amssymb,amsthm,mathtools}
\usepackage{aliascnt}
\usepackage{microtype}
\usepackage{placeins}
\usepackage{tikz}
\usetikzlibrary{arrows.meta,calc,positioning}
\usepackage[hidelinks]{hyperref}
\usepackage[nameinlink,capitalise,noabbrev]{cleveref}

\hypersetup{
  pdftitle={Local decisions, diffusive influence, and lower bounds for graphical balanced allocation},
  pdfauthor={Obinna Okechukwu},
  pdfsubject={Graphical balanced allocation and lower bounds for local rules},
  pdfkeywords={graphical balanced allocation, power of two choices, load balancing, local algorithms, coupling, martingales, stochastic growth, Gaussian free field}
}

\newtheorem{theorem}{Theorem}[section]
\newaliascnt{proposition}{theorem}
\newtheorem{proposition}[proposition]{Proposition}
\aliascntresetthe{proposition}
\newaliascnt{lemma}{theorem}
\newtheorem{lemma}[lemma]{Lemma}
\aliascntresetthe{lemma}
\newaliascnt{corollary}{theorem}
\newtheorem{corollary}[corollary]{Corollary}
\aliascntresetthe{corollary}
\theoremstyle{remark}
\newaliascnt{remark}{theorem}
\newtheorem{remark}[remark]{Remark}
\aliascntresetthe{remark}
\newaliascnt{problem}{theorem}
\newtheorem{problem}[problem]{Problem}
\aliascntresetthe{problem}
\crefname{theorem}{Theorem}{Theorems}
\crefname{proposition}{Proposition}{Propositions}
\crefname{lemma}{Lemma}{Lemmas}
\crefname{corollary}{Corollary}{Corollaries}
\crefname{remark}{Remark}{Remarks}
\crefname{problem}{Problem}{Problems}

\newcommand{\E}{\mathbb{E}}
\newcommand{\Pp}{\mathbb{P}}
\newcommand{\Z}{\mathbb{Z}}
\newcommand{\one}{\mathbf{1}}
\newcommand{\dist}{\operatorname{dist}}
\newcommand{\Gap}{\operatorname{Gap}}
\newcommand{\Var}{\operatorname{Var}}
\newcommand{\clip}{\operatorname{clip}}
\newcommand{\cM}{\mathcal{M}}
\newcommand{\ip}[2]{\langle #1,#2\rangle}
\newcommand{\norm}[1]{\lVert #1\rVert}
\newcommand{\doi}[1]{\href{https://doi.org/#1}{doi:#1}}

\title{Local decisions, diffusive influence, and lower bounds for graphical balanced allocation}
\author{Obinna Okechukwu}
\date{}

\begin{document}
\maketitle

\begin{abstract}
In graphical two-choice allocation, each arriving ball is assigned to one endpoint of a random edge. We study rules whose decision is a monotone function of the two endpoint loads, allowing edge-dependent thresholds and fresh randomization. Such a rule has an exact unit-discrepancy coupling: adding one ball to the initial state produces one tagged discrepancy at every later time. We represent the tag by conditional-expectation projections on the marked edge space and obtain diffusive displacement bounds. A transport-volume inequality then converts slow propagation of influence into lower bounds for the load gap. On the cycle with $n$ vertices, from every initial distribution and at every physical time $t\ge 1/n$, the expected gap is at least a constant times $\min\{\sqrt n,t^{1/4}\}$, and the gap exceeds this scale with probability at least $1/8$. After exactly $k\ge1$ allocations, the corresponding scale is $\min\{\sqrt n,(k/n)^{1/4}\}$. No stationarity, symmetry, recurrence, or moment assumption is used. A smoothed threshold rule in the same class has expected gap $O(\sqrt n\log n)$ up to any fixed polynomial time horizon, so the saturated cycle bound is sharp within the class up to a logarithmic factor. The general inequality also yields a lower bound of order $\sqrt{L/K}$ on the $L\times K$ rectangular torus $C_L\square C_K$; combined with a strategy-independent logarithmic bound, this gives order $\sqrt{L/K}+\log(LK)$. These results separate endpoint-local rules from global-information strategies that achieve polylogarithmic gaps on cycles. All numbered results are verified in Lean.
\end{abstract}

\noindent\textbf{Keywords.}
Graphical balanced allocation; power of two choices; local algorithms; coupling; martingales; stochastic growth; Gaussian free field.

\medskip
\noindent\textbf{2020 Mathematics Subject Classification.}
Primary 60K35; Secondary 60J27, 68W20, 60G44, 05C81.

\section{Introduction}

The graphical two-choice process is an online load-balancing model on a fixed graph. Each arriving ball samples an edge and must be placed at one of its endpoints. The greedy rule places the ball at the less loaded endpoint, breaking ties uniformly; with the orientation convention used below, it is the choice $p_e(d)=\one_{\{d<0\}}+\tfrac12\one_{\{d=0\}}$. We study fixed rules whose decision is a monotone function of the two endpoint loads, allowing edge-dependent thresholds and fresh randomization. We call these \emph{endpoint-local monotone rules}. Our main result shows that this information restriction forces a square-root gap on the cycle: influence propagates only diffusively, so a macroscopic imbalance persists long enough to force a gap of order $\sqrt n$.

Physical time is normalized so that every edge rings at rate one; hence the total event rate on the $n$-cycle is $n$. For a load vector $x$, write
\[
 \Gap(x)=\max_v x_v-\min_v x_v.
\]

\begin{theorem}[Continuous-time cycle lower bound]\label{thm:intro-cycle}
There is an absolute constant $c>0$ with the following property. Let $n\ge3$, let an endpoint-local monotone rule run on $C_n$, and let the initial integer-valued load vector have an arbitrary distribution independent of the future clocks and random marks. Then, for every $t\ge1/n$,
\[
 \E\Gap(X_t)\ge c\min\{\sqrt n,t^{1/4}\},
\]
and
\[
 \Pp\!\left(\Gap(X_t)\ge c\min\{\sqrt n,t^{1/4}\}\right)\ge\frac18.
\]
One may take $c=(1024\pi\sqrt3)^{-1}$.
\end{theorem}

The theorem is transient and uniform over the initial state. In particular, it applies from the flat configuration, and the saturated lower bound holds at every $t\ge n^2$. If the process modulo common translations admits an invariant probability measure $\pi$, then
\[
 \E_\pi\Gap\ge c\sqrt n,
 \qquad
 \pi(\Gap\ge c\sqrt n)\ge\frac18.
\]
No uniqueness or integrability property of $\pi$ is needed.

There is also a deterministic-event version. Let $X^{[k]}$ denote the load vector after exactly $k$ allocations.

\begin{theorem}[Cycle lower bound at fixed allocation counts]\label{thm:intro-discrete}
There is an absolute constant $c'>0$ with the following property. Let $n\ge3$, let an endpoint-local monotone rule run on $C_n$, and let the initial integer-valued load vector have an arbitrary distribution independent of the future marks. Then, for every $k\ge1$,
\[
 \E\Gap(X^{[k]})\ge c'\min\{\sqrt n,(k/n)^{1/4}\},
\]
and
\[
 \Pp\!\left(\Gap(X^{[k]})\ge c'\min\{\sqrt n,(k/n)^{1/4}\}\right)\ge\frac18.
\]
One may take $c'=(16384\pi\sqrt3)^{-1}$.
\end{theorem}

Thus the square-root gap holds after every $k\ge n^3$. Within the endpoint-local monotone class the bound is sharp up to a logarithmic factor: \cref{cor:smoothed-cycle} shows that a smoothed threshold rule, started from the flat profile, has
\[
 \E\Gap(X_t)\le 300\sqrt n\log n+2
 \qquad(0\le t\le n^4),
\]
and the same bound holds after every $k\le n^{10}$ allocations. Whether greedy allocation itself has gap $O(\sqrt n)$ remains open.

The general transport-volume inequality also yields
\[
 \E\Gap(X_t)\gtrsim \sqrt{L/K}
 \qquad (3\le K\le L,\ t\ge L^2)
\]
on the rectangular torus $C_L\square C_K$. Combined with a strategy-independent logarithmic bound (\cref{prop:universal-log}), this gives
\[
 \E\Gap(X_t)\gtrsim \sqrt{L/K}+\log(LK),
\]
with the same lower scale holding with probability at least $1/8$ (\cref{cor:torus-combined}). No attempt is made to optimize the numerical constants.

\paragraph{Related work.}
The classical two-choice process reduces the maximum load from order $\log n/\log\log n$ to order $\log\log n$ when $n$ balls are allocated to $n$ bins \cite{AzarBroderKarlinUpfal1999}. Kenthapadi and Panigrahy introduced the graphical version \cite{KenthapadiPanigrahy2006}. Peres, Talwar, and Wieder bounded the heavily loaded graphical process in terms of edge expansion \cite{PeresTalwarWieder2015}; on a cycle, their estimate gives $O(n\log n)$. For the synchronous equilibrium process on a regular graph, Olesker-Taylor, Sauerwald, and Zanetti proved that the edge-average of the expected load difference is at most two \cite{OleskerTaylorSauerwaldZanetti2026}, which gives the current $O(n)$ equilibrium upper bound on the cycle.

Bansal and Feldheim showed that more global information changes the picture: their strategy compares average loads on fixed vertex sets and obtains a polylogarithmic gap on cycles and tori \cite{BansalFeldheim}. They also proved a strategy-independent $\Omega(\log n)$ lower bound for every graph, together with an $\Omega(d/k)$ term, where $k$ denotes the edge-connectivity of the $d$-regular graph. On the cycle the universal lower bound is only $\Omega(\log n)$, so no strategy-independent argument can yield the square-root scale. Their paper identifies polynomial behavior for greedy allocation on the cycle as a natural conjecture and reports simulations consistent with $\sqrt n$-scale behavior. The locality separation here is about information, not the number of bins modified: both procedures place one ball at one endpoint, but only the endpoint-local rule is covered by our lower bound.

Restrictions on available information have also been studied in complete-graph models, including binary load queries \cite{LosSauerwald2022} and thinning \cite{FeldheimGurelGurevich2021}. In the lightly loaded i.i.d.\ regime, Bansal, Prabhu, Singla, and Sundaram showed that greedy can lose a nearly logarithmic factor against the offline optimum on mildly irregular base graphs, while a decomposition-based threshold rule using knowledge of the base graph is $O(\log\log n)$-competitive \cite{BansalPrabhuSinglaSundaram2026}. Our results concern the heavily loaded regime, where the limitation of rules that decide from the two endpoint loads is geometric.

The dynamic averaging process studied by Alistarh, Nadiradze, and Sabour is related but different: after load is introduced, the two endpoint loads are averaged. They proved an $O(\sqrt n\log n)$ expected upper bound on cycles. Their Theorem~2 proves $\lim_{t\to\infty}\E[\Gap(t)^2]=\Omega(n\E[W^2])$, where $W$ is the arriving load, and its proof uses the $\lfloor n/2\rfloor$-hop potential \cite[Theorem~2]{AlistarhNadiradzeSabour2022}. Kraizberg subsequently proved an $O_d(\sqrt n)$ expected upper bound on every $d$-regular graph, the sharp $O(\log n)$ upper bound on the two-dimensional torus, and a matching cycle lower bound when the arriving loads are bounded away from zero \cite{Kraizberg2026}. These averaging results do not imply a lower bound for graphical allocation because the update maps and their responses are different.

The exponents in \cref{thm:intro-cycle,thm:intro-discrete} agree with the one-dimensional Edwards--Wilkinson scaling prediction: roughness grows like time to the power $1/4$, saturates at spatial scale to the power $1/2$, and has dynamic exponent $2$ \cite{EdwardsWilkinson1982,FamilyVicsek1985,Family1986}. Family's random-deposition-with-surface-relaxation model is a standard discrete representative of that universality class. We use this comparison only as motivation; the endpoint-local allocation processes considered here are neither assumed nor proved to belong to an Edwards--Wilkinson universality class.

\paragraph{Proof overview.}
Influence spreads diffusively. If the gap is at most $M$ with probability at least $7/8$, then a clipped contrast at spatial scale $r$ has derivative of order $1/M$ on most terminal configurations. Its positive and negative endpoint responses remain separated for time of order $r^2$, and at each lag Cauchy--Schwarz forces response energy of order $1/(M^2r)$. Since the clipped test has variance at most one, integrating this cost over $r^2$ lags gives $M^2\gtrsim r$.

The proof begins with a structural characterization. A deterministic edge update preserves a single unit discrepancy under the synchronous coupling if and only if its decision uses only the two endpoint loads and is monotone in their difference. For randomized rules, an auxiliary uniform mark realizes each update as a deterministic monotone threshold rule. This synchronous coupling is the basic coupling familiar from attractive interacting particle systems, and the extra unit may be viewed as a second-class particle \cite{Liggett1985}. In exclusion and related deposition models, exact identities relate height or current fluctuations to the spread of a second-class particle \cite{FerrariFontes1994,BalazsSeppalainen2007,BalazsSeppalainen2010}, and in the Helffer--Sj\"ostrand approach to gradient interface models, covariances are represented through a random walk in a dynamic environment \cite{HelfferSjostrand1994,NaddafSpencer1997,GiacominOllaSpohn2001}. Those results rely on explicit or reversible invariant measures. No such structure is assumed or used here, and we obtain inequalities from a finite-horizon variance budget instead. The ingredients specific to the present setting are the exact characterization of when the discrepancy remains a single tag, its marked Palm representation, and the transport-volume inequality, which combines a finite-horizon variance budget with a volume count in the manner of on-diagonal heat-kernel lower bounds, now for a nonlinear, non-reversible dynamics.

Conditioned on the observed load path, the tag is represented by local conditional-expectation projections on the marked edge space. Writing $J_s$ for its position after lag $s$, a product-of-projections estimate gives, on the cycle,
\[
 \E\dist(J_s,J_0)^2\le 2\pi^2s+2.
\]
We then average a translated family of clipped two-point observables before estimating their response. The positive endpoint contribution can be lost only when the tag exits a protected ball; the negative endpoint contribution can appear only after a displacement of the same order. The finite-horizon martingale variance identity converts this persistence into
\[
 M^2\ge \frac mN\int_0^T
 \frac{[1-p-2q(s)]_+^2}{V_d(R(s))}\,ds,
\]
where $p$ is the gap tail at scale $M$, $q(s)$ is a displacement tail, and $V_d(R)$ is the largest comparison-ball volume. Varying $R(s)$ within the same terminal observable family handles graphs whose ball volumes change growth rate, such as rectangular tori.

To the best of our knowledge, the $\Omega(\sqrt n)$ lower bound for greedy graphical allocation on the cycle, and the uniform finite-time result for the endpoint-local monotone class, have not appeared previously. The proof is self-contained apart from elementary martingale facts and Hall's marriage theorem; when Hall's theorem is used, we verify its condition directly.

\paragraph{Formal verification.}
All 22 numbered theorems, lemmas, propositions, and corollaries have been formalized and verified in Lean~4 with Mathlib~\cite{OkechukwuLean2026}. The accepted axioms are \texttt{propext}, \texttt{Classical.choice}, and \texttt{Quot.sound}; no additional axioms or unproved assumptions are used. The coverage guide in the repository maps each result to its Lean declaration.

The paper is organized as follows. \Cref{sec:rules} defines endpoint-local rules and proves the unit-discrepancy characterization. \Cref{sec:palm} constructs the marked Palm tag. \Cref{sec:diffusion} proves displacement estimates. \Cref{sec:transport} establishes the transport-volume inequality. \Cref{sec:cycle} derives the continuous- and discrete-time cycle theorems. \Cref{sec:graphs} treats Hilbert embeddings, finite-width cylinders, and rectangular tori, and proves a strategy-independent logarithmic bound. \Cref{sec:smoothed} shows that a smoothed threshold rule nearly attains the cycle bound. \Cref{sec:open} records open problems.

\section{Endpoint-local monotone rules}\label{sec:rules}

Let $G=(V,E)$ be a finite simple graph with $N=|V|\ge2$ and $m=|E|\ge1$. Every edge rings according to an independent rate-one Poisson process. Fix an orientation $(u,v)$ of each edge $e$. When $e$ rings, draw an independent random variable $U$ uniformly distributed on $(0,1)$ and allocate a ball to $u$ precisely when
\[
 U\le p_e(x_u-x_v),
\]
where $p_e:\Z\to[0,1]$ is nonincreasing. Otherwise allocate to $v$. The functions $p_e$ may differ between edges. Constant functions are allowed.

The model is invariant under common integer translations of the load vector. The \emph{normalized process} is the induced Markov process on the quotient $\Z^V/\Z\mathbf{1}$; an invariant normalized law is an invariant probability measure on this quotient. Expectations of nonnegative quantities such as $\Gap$ take values in $[0,\infty]$; no integrability is assumed. We write $P_t$ for the continuous-time semigroup, $X^{[h]}$ for the load vector after exactly $h$ allocation events, and $\mathsf P$ for the one-event transition operator. For a function $f$, set
\[
 D_w f(x)=f(x+e_w)-f(x).
\]
The total rate at which vertex $w$ receives a ball is denoted by $\lambda_w(x)$. Thus
\[
 0\le \lambda_w(x)\le \deg(w),
 \qquad
 \sum_{w\in V}\lambda_w(x)=m.
 \tag{2.1}\label{eq:rates}
\]

The following proposition identifies exactly when a deterministic edge update carries a one-ball perturbation as a single discrepancy.

\begin{proposition}[Characterization of unit-discrepancy updates]\label{prop:characterization}
Fix an edge $\{u,v\}$. Let $s:\Z^V\to\{u,v\}$ be a translation-invariant deterministic selector, and let
\[
 T(x)=x+e_{s(x)}.
\]
Then
\[
 T(x+e_z)-T(x)\in\{e_w:w\in V\}
 \tag{2.2}\label{eq:unit-property}
\]
for every profile $x$ and every vertex $z$ if and only if $s$ depends only on $x_u-x_v$ and the function
\[
 d\longmapsto \one_{\{s(x)=u\}},\qquad d=x_u-x_v,
\]
is nonincreasing.
\end{proposition}

\begin{proof}
Assume first that \eqref{eq:unit-property} holds. If $z\notin\{u,v\}$ and $s(x+e_z)\ne s(x)$, then
\[
 T(x+e_z)-T(x)=e_z+e_{s(x+e_z)}-e_{s(x)},
\]
which has three nonzero coordinates. Hence increasing an off-edge coordinate cannot change the selection. Applying the same conclusion to $x-e_z$ shows that decreasing that coordinate cannot change the selection either, so $s$ is independent of all off-edge coordinates.

Suppose next that $s(x)=v$ but $s(x+e_u)=u$. Then
\[
 T(x+e_u)-T(x)=2e_u-e_v,
\]
contrary to \eqref{eq:unit-property}. Thus raising the load of $u$ cannot switch the selection toward $u$. The analogous statement holds for $v$. Translation invariance now shows that the selector is a nonincreasing function of $x_u-x_v$.

Conversely, an off-edge perturbation changes no decision. If $z$ is an endpoint not selected at $x$, monotonicity prevents the update from switching toward the raised endpoint, so the discrepancy remains at $z$. If $z=s(x)$, then raising $x_z$ either leaves the selection unchanged or switches it to the opposite endpoint. In the latter case the discrepancy moves to that endpoint. In every case \eqref{eq:unit-property} holds.
\end{proof}

For a randomized rule, the common uniform mark $U$ realizes the update as a deterministic threshold rule for each fixed $U$. Synchronously coupling two processes from $x$ and $x+e_v$ therefore gives a vertex-valued tag such that
\[
 X_t^{x+e_v}-X_t^x=e_{J_t}
 \tag{2.3}\label{eq:one-ball-coupling}
\]
at all physical times. We write $\E_{x,v}$ for expectation under this coupling, started from $x$ and $x+e_v$ with $J_0=v$. If $N_t$ is the number of allocation events by time $t$, then $J_h$ denotes the tag after exactly $h$ events and $J_t=J_{N_t}$ denotes its physical-time position. Consequently, for every bounded translation-invariant function $f$,
\[
 D_vP_tf(x)=\E_{x,v}D_{J_t}f(X_t).
 \tag{2.4}\label{eq:derivative-tag}
\]
The expectation on the right averages over all future clocks and marks. The same coupling in event time gives
\[
 D_v\mathsf P^h f(x)=\E_{x,v}D_{J_h}f(X^{[h]}),
 \qquad h\in\Z_{\ge0}.
 \tag{2.5}\label{eq:derivative-tag-discrete}
\]

\begin{remark}[Scope]\label{rem:scope}
The characterization explains why a strategy may place one ball at one endpoint and still fall outside the theorem. If a decision reads an off-edge load, raising that load can change the selected endpoint and produce a discrepancy of the form $e_z+e_u-e_v$. If endpoint monotonicity fails, a perturbation can produce $2e_z-e_w$. The global set-average strategy of Bansal and Feldheim uses the first mechanism and is therefore not endpoint-local in the sense of \cref{prop:characterization}.
\end{remark}

\section{The marked Palm tag}\label{sec:palm}

The initial location of the discrepancy is biased by the instantaneous allocation rates. Conditional on the current profile $x$, the law $\lambda_w(x)/m$ is the location law of the next allocation; equivalently, it is the Palm distribution of the allocation point process. The corresponding construction is easiest on the marked edge space
\[
 \cM=E\times(0,1).
\]
Let $\rho$ be the probability measure that assigns mass $1/m$ to each edge fiber and Lebesgue measure within the fiber. For a mark $a=(e,U)$, write $\sigma_x(a)$ for the selected endpoint at profile $x$, and define the selection cells
\[
 C_w(x)=\{a\in\cM:\sigma_x(a)=w\}.
\]
Then
\[
 \rho(C_w(x))=\frac{\lambda_w(x)}m.
 \tag{3.1}\label{eq:cell-mass}
\]

When a ball is allocated to $j$, endpoint monotonicity implies that only marks that formerly selected $j$ can change their selected endpoint, and those marks can move only to a neighbor of $j$. The next proposition packages this observation as a sequence of conditional-expectation projections.

\begin{proposition}[Conditional mark representation]\label{prop:palm}
Initialize the tag at a profile $x$ according to
\[
 \Pp(J_0=w\mid X_0=x)=\frac{\lambda_w(x)}m.
 \tag{3.2}\label{eq:palm-init}
\]
Condition on the initial profile, the event times, and the selected base vertices, but not on the hidden edge marks. There is an auxiliary process $(B_h)_{h\ge0}$ on $\cM$ such that
\[
 \bigl(\sigma_{X^{[h]}}(B_h)\bigr)_{h\ge0}
\]
has the same conditional law as the event-indexed discrepancy tag $(J_h)_{h\ge0}$. The initial mark $B_0$ has law $\rho$, and each transition of $B$ is an orthogonal conditional-expectation projection on $L^2(\rho)$. The transition preserves $\rho$ and acts nontrivially only on edge fibers incident to the closed neighborhood of the selected vertex. Consequently, at physical time $t$,
\[
 \Pp(J_t=w\mid\text{selected base path})
 =\frac{\lambda_w(X_t)}m.
 \tag{3.3}\label{eq:conditional-palm}
\]
For every bounded measurable function $F$ of a profile and a tag,
\[
 \sum_v\lambda_v(x)\E_{x,v}F(X_t,J_t)
 =\E_x\sum_w\lambda_w(X_t)F(X_t,w).
 \tag{3.4}\label{eq:palm-link}
\]
\end{proposition}

\begin{proof}
Suppose the base update selects $j$. By endpoint monotonicity, only marks that formerly selected $j$ can change their selected endpoint when $x_j$ is raised, and such a mark can move only to a neighbor of $j$. Let
\[
 U_j=\bigcup_{w\in N[j]}C_w(x),
\]
where $N[j]$ is the closed neighborhood of $j$. The preceding observation shows that $U_j$ is also the union of the new cells indexed by $N[j]$.

Given an old auxiliary mark $b$, first read its new selection cell. If $b\in U_j$, resample it from $\rho$ conditioned on that new cell; if $b\notin U_j$, leave it fixed. This kernel is conditional expectation onto the sigma algebra that distinguishes every point outside $U_j$ and, inside $U_j$, remembers only the new selection cell. It is therefore self-adjoint, idempotent, and $\rho$-preserving. A null cell is never reached under the conditioned law; the kernel may be defined arbitrarily on its null subset.

We prove the conditional-law statement by induction, maintaining the following invariant: conditional on the selected path up to the current event and on the current tag value $w$, the auxiliary mark has law $\rho$ conditioned on the current cell $C_w$. The invariant holds initially because $B_0\sim\rho$ and the tag is its selected endpoint. Suppose it holds before an update selecting $j$. If the old tag $w$ is not $j$, it keeps its value, since raising $x_j$ cannot move a mark toward $j$. If $w\notin N[j]$, its cell is unchanged; if $w$ is a neighbor of $j$, its cell can only grow, and the auxiliary mark is resampled within the new cell $C_w(x+e_j)$. If the old tag is $j$, reading the new cell gives
\[
 \Pp(J'=w\mid x,j,J=j)
 =\frac{\rho(C_j(x)\cap C_w(x+e_j))}{\rho(C_j(x))}.
 \tag{3.5}\label{eq:tag-transition}
\]
The subsequent resampling makes the auxiliary mark conditionally $\rho$-uniform in that new cell. Thus the invariant survives: marks in $U_j$ are resampled within their new cells, while cells outside $U_j$ do not change.

The hidden event mark, conditional on the base selection of $j$, has law $\rho(\,\cdot\mid C_j(x))$, so \eqref{eq:tag-transition} is also the transition law of the discrepancy tag in the synchronous coupling. Independence of the event marks shows that conditioning on the entire selected path introduces no further bias beyond these successive cell conditions. The induction therefore proves equality in conditional law. Since every auxiliary transition preserves $\rho$, \eqref{eq:conditional-palm} follows from \eqref{eq:cell-mass}. Integrating against $F$ gives \eqref{eq:palm-link}. The same construction and identities apply after any prescribed number of allocation events.
\end{proof}

\begin{figure}[t]
\centering
\begin{tikzpicture}[x=1cm,y=1cm,font=\small,>=Latex]
  \def\bw{2.25}
  \def\gap{0.32}
  \foreach \q/\x/\lab in {
    0/0/{e_{j-2,j-1}},
    1/2.57/{e_{j-1,j}},
    2/5.14/{e_{j,j+1}},
    3/7.71/{e_{j+1,j+2}}}{
      \node at ({\x+1.125},2.55) {$\lab$};
      \draw (\x,1.55) rectangle ({\x+\bw},2.05);
      \draw (\x,0) rectangle ({\x+\bw},0.5);
  }
  \node[anchor=east] at (-0.2,1.8) {before};
  \node[anchor=east] at (-0.2,0.25) {after};
  \node at (4.98,1.03) {$x_j\mapsto x_j+1$};

  \draw (1.25,1.55)--(1.25,2.05);
  \node at (0.62,1.8) {$j-2$}; \node at (1.75,1.8) {$j-1$};
  \draw (1.25,0)--(1.25,0.5);
  \node at (0.62,0.25) {$j-2$}; \node at (1.75,0.25) {$j-1$};
  \draw (8.73,1.55)--(8.73,2.05);
  \node at (8.22,1.8) {$j+1$}; \node at (9.46,1.8) {$j+2$};
  \draw (8.73,0)--(8.73,0.5);
  \node at (8.22,0.25) {$j+1$}; \node at (9.46,0.25) {$j+2$};

  \fill[black!12] (2.57,1.55) rectangle (3.35,2.05);
  \fill[black!55] (3.35,1.55) rectangle (4.82,2.05);
  \draw (3.35,1.55)--(3.35,2.05);
  \node at (2.96,1.8) {$j-1$}; \node[white] at (4.08,1.8) {$j$};
  \fill[black!55] (5.14,1.55) rectangle (6.60,2.05);
  \fill[black!12] (6.60,1.55) rectangle (7.39,2.05);
  \draw (6.60,1.55)--(6.60,2.05);
  \node[white] at (5.87,1.8) {$j$}; \node at (7.00,1.8) {$j+1$};

  \fill[black!12] (2.57,0) rectangle (4.05,0.5);
  \fill[black!55] (4.05,0) rectangle (4.82,0.5);
  \draw (4.05,0)--(4.05,0.5);
  \node at (3.31,0.25) {$j-1$}; \node[white] at (4.44,0.25) {$j$};
  \fill[black!55] (5.14,0) rectangle (5.91,0.5);
  \fill[black!12] (5.91,0) rectangle (7.39,0.5);
  \draw (5.91,0)--(5.91,0.5);
  \node[white] at (5.52,0.25) {$j$}; \node at (6.65,0.25) {$j+1$};

  \draw[->,thick] (3.70,1.43) -- (3.70,0.62);
  \draw[->,thick] (6.25,1.43) -- (6.25,0.62);
  \node[align=center] at (4.98,-0.48)
    {marks lost by $C_j$ are absorbed by $C_{j-1}$ or $C_{j+1}$};
  \draw[dashed,rounded corners] (-0.10,-0.78) rectangle (10.08,2.82);
  \node[anchor=west] at (10.18,1.03) {affected edge fibers};
\end{tikzpicture}
\caption{Selection-cell repartition after a ball is allocated to $j$. Each box is one edge fiber of the marked space, partitioned according to the endpoint selected by the corresponding marks. Raising $x_j$ can only shrink the portions selecting $j$ on the two incident edges. Conditional resampling averages within the new cells in the dashed union and is the identity outside it.}
\label{fig:cells}
\end{figure}
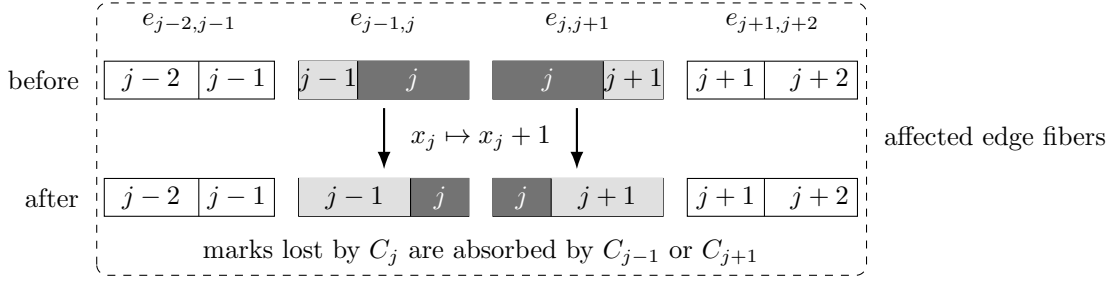

\Cref{fig:cells} depicts the cycle case. The proportions of the intervals are rule-dependent; only the nesting forced by monotonicity is used. The outer two fibers are needed because the new cells of $j-1$ and $j+1$ include their other incident edges. This accounts for the four-fiber bound in the displacement estimate below.

\section{Diffusive displacement}\label{sec:diffusion}

The conditional mark process is generally time-inhomogeneous, and its projection operators need not commute. The following estimate controls their product without reordering them.

\begin{lemma}[Products of projections]\label{lem:projections}
Let $Q_1,\ldots,Q_k$ be orthogonal projections on a real or complex Hilbert space. Then
\[
 \operatorname{Re}\ip{f}{(I-Q_k\cdots Q_1)f}
 \le 2\sum_{j=1}^k\norm{(I-Q_j)f}^2.
 \tag{4.1}\label{eq:projection-product}
\]
If $(B_j)$ is the corresponding sequence of projection kernels, started from their common invariant probability $\rho$, then
\[
 \E\norm{f(B_k)-f(B_0)}^2
 \le4\sum_{j=1}^k\norm{(I-Q_j)f}_{L^2(\rho)}^2.
 \tag{4.2}\label{eq:projection-chain}
\]
The second assertion remains valid for Hilbert-valued $f$.
\end{lemma}

\begin{proof}
Set $v_0=f$ and $v_j=Q_jv_{j-1}$. Define
\[
 A=\operatorname{Re}\ip{f}{f-v_k},\qquad
 S=\sum_j\norm{(I-Q_j)f}^2,
 \qquad
 D=\sum_j\norm{(I-Q_j)v_{j-1}}^2.
\]
Pythagoras gives $D=\norm f^2-\norm{v_k}^2$. Since
\[
 f-v_k=\sum_j(I-Q_j)v_{j-1}
\]
and $(I-Q_j)v_{j-1}$ is orthogonal to the range of $Q_j$,
\[
 A
 =\sum_j\operatorname{Re}\ip{(I-Q_j)f}{(I-Q_j)v_{j-1}}
 \le\sqrt{SD}.
\]
Also
\[
 2A-D=\norm{f-v_k}^2\ge0,
\]
and contraction gives $A\ge0$. If $A>0$, then $A^2\le SD\le2SA$, which proves \eqref{eq:projection-product}; the case $A=0$ is immediate.

For the Markov chain, expand the square in \eqref{eq:projection-chain}. The cross term is the quadratic form of the forward product, whose adjoint is the displayed reverse product and has the same real part. Apply \eqref{eq:projection-product}. The Hilbert-valued case is the same argument in $L^2(\rho;\mathcal H)$.
\end{proof}

Inequality \eqref{eq:projection-product} is closely related to the quantum union bound for sequential projective measurements. Indeed, the same proof gives
\[
 \norm f^2-\norm{Q_k\cdots Q_1f}^2
 \le4\sum_{j=1}^k\norm{(I-Q_j)f}^2,
\]
which is Gao's projective-measurement union bound in this Hilbert-space form \cite{Gao2015}; see also \cite{ODonnellVenkateswaran2022} for a short proof.

On the cycle, the first Fourier mode of the edge midpoints converts \cref{lem:projections} into a diffusive estimate.

\begin{proposition}[Diffusive displacement on the cycle]\label{prop:cycle-diffusion}
For every endpoint-local monotone rule on $C_n$, every starting profile, and the initialization \eqref{eq:palm-init}, the tag after exactly $h$ allocations satisfies
\[
 \E\dist_{C_n}(J_h,J_0)^2\le \frac{2\pi^2h}{n}+2.
 \tag{4.3}\label{eq:cycle-diffusion-events}
\]
Consequently, at physical time $s$,
\[
 \E\dist_{C_n}(J_s,J_0)^2\le2\pi^2s+2.
 \tag{4.4}\label{eq:cycle-diffusion}
\]
\end{proposition}

\begin{proof}
Condition on a base path with $h$ allocation events. Give each marked edge its cyclic midpoint $z$ and set
\[
 f(a)=e^{2\pi iz(a)/n}.
\]
A selection cell at a vertex uses at most the two adjacent edge midpoints, whose cyclic distance is one. Its conditional variance is at most $\sin^2(\pi/n)$. If the vertex selected at update $j$ is $w$, the affected union is contained in the four edge fibers incident to $w-1,w,$ or $w+1$. Hence
\[
 \norm{(I-Q_j)f}_{L^2(\rho)}^2
 \le\frac4n\sin^2(\pi/n)
\]
for every $1\le j\le h$. By \eqref{eq:projection-chain},
\[
 \E|f(B_h)-f(B_0)|^2
 \le\frac{16h}{n}\sin^2(\pi/n).
\]
If $\delta\in[0,n/2]$ is the cyclic distance between the two edge midpoints, then
\[
 |f(B_h)-f(B_0)|=2\sin(\pi\delta/n)\ge4\delta/n.
\]
Thus
\[
 \E\delta^2\le nh\sin^2(\pi/n)\le\pi^2h/n.
\]
Each selected endpoint lies within distance $1/2$ of its edge midpoint. Therefore
\[
 \E\dist(J_h,J_0)^2\le2\pi^2h/n+2.
\]
The number of events by physical time $s$ is Poisson with mean $ns$. Averaging over it proves \eqref{eq:cycle-diffusion}.
\end{proof}

The same argument works with a Hilbert-space coordinate on a general graph. We postpone that form to \cref{sec:graphs}, where it is used.

\section{The transport-volume inequality}\label{sec:transport}

The martingale variance identity supplies the finite energy budget. The terminal time is fixed, so no stationary law or long-time limit enters.

\begin{lemma}[Finite-horizon response energy]\label{lem:variance-budget}
Let $f$ be bounded and invariant under common translations. For every initial law and every $t\ge0$,
\[
 \int_0^t\E\sum_v\lambda_v(X_{t-s})
      |D_vP_sf(X_{t-s})|^2\,ds
 \le \Var(f(X_t)).
 \tag{5.1}\label{eq:variance-budget}
\]
\end{lemma}

\begin{proof}
The process
\[
 M_u=P_{t-u}f(X_u),\qquad 0\le u\le t,
\]
is a bounded martingale. At an allocation to $v$, its jump is $D_vP_{t-u}f(X_{u-})$, and such allocations occur at rate $\lambda_v(X_{u-})$. The expected predictable quadratic variation is therefore the left side of \eqref{eq:variance-budget}, after the change of variables $s=t-u$. The martingale isometry identifies this expectation with
\[
 \E f(X_t)^2-\E(P_tf(X_0))^2
 \le \E f(X_t)^2-(\E f(X_t))^2.
\]
All finite-horizon terms are integrable because $f$ is bounded and the total jump rate is $m$.
\end{proof}

Let $d$ be a metric or pseudometric on $V$, and write
\[
 B(v,R)=\{w:d(v,w)\le R\},
 \qquad
 V_d(R)=\max_v|B(v,R)|.
\]
Fix a permutation $\psi$ of $V$. For $0\le s\le T$, let $R(s)>0$ satisfy
\[
 d(i,\psi(i))\ge2R(s)
 \qquad\text{for every }i.
 \tag{5.2}\label{eq:separation}
\]
Let $R$ and $q$ be measurable functions on $[0,T]$, and suppose that $q(s)$ bounds the Palm-tag displacement tail uniformly in the starting profile:
\[
 \Pp(d(J_s,J_0)\ge R(s))\le q(s).
 \tag{5.3}\label{eq:tail-envelope}
\]

\begin{figure}[!ht]
\centering
\begin{tikzpicture}[>=Latex,font=\small]
  \def\rad{1.62}
  \def\shadeStart{112.5}
  \def\shadeEnd{247.5}
  \def\orad{1.88}

  \begin{scope}[shift={(-2.70,0)}]
    \node[font=\small] at (0,2.38) {(a) Lost positive credit};
    \draw[thick] (0:\rad) arc (0:360:\rad);
    \foreach \theta in {0,22.5,...,337.5}{
      \node[circle,fill=black,inner sep=1.05pt] at (\theta:\rad) {};
    }
    \draw[line width=5pt,black!28]
      (\shadeStart:\rad) arc (\shadeStart:\shadeEnd:\rad);
    \node[circle,fill=white,draw=black,inner sep=2pt,label=right:$i$]
      (left-i) at (180:\rad) {};
    \node[circle,fill=white,draw=black,inner sep=1.6pt,
      label={[label distance=2pt]below:$u$}]
      (left-out) at (90:\rad) {};
    \draw[->,very thick,dashed]
      (95:\orad) arc[start angle=95,end angle=172,radius=\orad];
    \node[anchor=east,fill=white,inner sep=1.5pt]
      at (-1.75,1.30) {$d(u,i)\ge R$};
  \end{scope}

  \begin{scope}[shift={(2.70,0)}]
    \node[font=\small] at (0,2.38) {(b) Negative credit};
    \draw[thick] (0:\rad) arc (0:360:\rad);
    \foreach \theta in {0,22.5,...,337.5}{
      \node[circle,fill=black,inner sep=1.05pt] at (\theta:\rad) {};
    }
    \draw[line width=5pt,black!28]
      (\shadeStart:\rad) arc (\shadeStart:\shadeEnd:\rad);
    \node[circle,fill=white,draw=black,inner sep=2pt,label=right:$i$]
      (right-i) at (180:\rad) {};
    \node[circle,fill=white,draw=black,inner sep=2pt,label=left:$i+r$]
      (right-j) at (0:\rad) {};
    \node[circle,fill=white,draw=black,inner sep=1.6pt,
      label={[label distance=2pt]above right:$v$}]
      (right-in) at (225:\rad) {};
    \draw[->,very thick,dotted]
      (230:\orad) arc[start angle=230,end angle=352,radius=\orad];
    \node[align=center,fill=white,inner sep=1.5pt]
      at (0,-2.20) {$d(v,i+r)\ge r-R\ge R$};
  \end{scope}
\end{tikzpicture}
\caption{The two exceptional histories in the averaged two-point test; the gray arc is $A_i=B(i,R)$. In (a), a tag that ends at $i$ but starts outside $A_i$ must travel at least $R$. In (b), a tag that starts in $A_i$ and ends at $i+r$ must travel at least $r-R\ge R$. Averaging over $i$ counts each terminal vertex once in each family.}
\label{fig:cycle-test}
\end{figure}
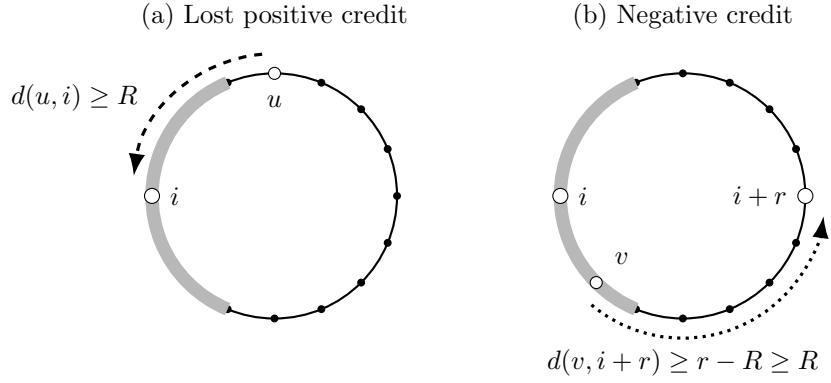
\FloatBarrier

\Cref{fig:cycle-test} shows the two exceptional histories in the averaged two-point test. If the tag remains inside $A_i$, then the positive derivative at $i$ and the negative derivative at $\psi(i)$ cannot cancel. While displacement by $R(s)$ is unlikely, the clipped contrast retains mean response on a region of size at most $V_d(R(s))$; conditional variance pays the square of that response. Integrating this cost over the available lags gives the transport-volume inequality.

\begin{theorem}[Transport-volume inequality]\label{thm:transport-volume}
Let $0\le T\le t$ and $M\ge1$. Put
\[
 p=\Pp(\Gap(X_t)>M-1).
\]
Under \eqref{eq:separation} and \eqref{eq:tail-envelope},
\[
 \boxed{
 M^2\ge\frac mN\int_0^T
 \frac{[1-p-2q(s)]_+^2}{V_d(R(s))}\,ds.}
 \tag{5.4}\label{eq:transport-volume}
\]
In particular, if $p\le1/8$ and $q(s)\le1/8$ throughout the interval, then
\[
 M^2\ge\frac m{4N}\int_0^T\frac{ds}{V_d(R(s))}.
 \tag{5.5}\label{eq:transport-volume-simple}
\]
\end{theorem}

\begin{proof}
For each $i\in V$, define the clipped contrast
\[
 f_i(x)=\clip_{[-1,1]}\!\left(\frac{x_i-x_{\psi(i)}}M\right).
\]
Its only nonzero finite differences are at $i$ and $\psi(i)$, and
\[
 D_if_i(x)\ge\frac1M\one_{\{\Gap(x)\le M-1\}},
 \qquad
 -\frac1M\le D_{\psi(i)}f_i(x)\le0.
 \tag{5.6}\label{eq:clipped-derivatives}
\]
For a fixed lag $s$, start the coupled tag at time $t-s$ according to \eqref{eq:palm-init}, and write $J_0,J_s$ for its positions during that interval. Set $A_i=B(i,R(s))$ and
\[
 S(s)=\frac1N\sum_i\E\sum_{v\in A_i}
 \lambda_v(X_{t-s})D_vP_sf_i(X_{t-s}).
\]
By the derivative identity \eqref{eq:derivative-tag} and the initialization \eqref{eq:palm-init},
\[
 S(s)=\frac mN\sum_i\E\bigl[
 \one_{\{J_0\in A_i\}}D_{J_s}f_i(X_t)
 \bigr].
 \tag{5.7}\label{eq:transport-palm}
\]
Without the restriction $J_0\in A_i$, the positive terminal contribution is
\[
 \frac mN\sum_i\E\bigl[
 \one_{\{J_s=i\}}D_if_i(X_t)
 \bigr]
 =\frac1N\E\sum_i\lambda_i(X_t)D_if_i(X_t)
 \ge\frac m{NM}(1-p).
 \tag{5.8}\label{eq:positive-credit}
\]
The equality uses the conditional Palm law \eqref{eq:conditional-palm}; the inequality uses \eqref{eq:clipped-derivatives} and the rate identity \eqref{eq:rates}.

Restricting to $J_0\in A_i$ removes positive credit only on
\[
 \{J_s=i,\ J_0\notin A_i\},
\]
which requires displacement at least $R(s)$. The only negative terminal derivative occurs at $\psi(i)$, and its contribution with $J_0\in A_i$ is supported on
\[
 \{J_s=\psi(i),\ J_0\in A_i\}.
\]
This event also requires displacement at least $R(s)$ because
\[
 d(i,\psi(i))-d(i,J_0)\ge R(s).
\]
When these events are summed over $i$, the terminal location is counted once in the first family because $i\mapsto i$ is a permutation and once in the second because $i\mapsto\psi(i)$ is a permutation. Since every derivative has magnitude at most $1/M$, their total cost is at most $2mq(s)/(NM)$. Hence
\[
 S(s)\ge\frac m{NM}[1-p-2q(s)].
 \tag{5.9}\label{eq:mean-response}
\]

Apply Cauchy--Schwarz to the measure assigning mass $N^{-1}\lambda_v(x)$ to triples $(i,x,v)$ with $v\in A_i$. Its total mass is at most
\[
 \frac1N\E\sum_v\lambda_v(X_{t-s})
       |\{i:v\in B(i,R(s))\}|
 \le\frac mN V_d(R(s)).
\]
Writing
\[
 \gamma_i(s,x)=\sum_v\lambda_v(x)|D_vP_sf_i(x)|^2,
\]
we obtain
\[
 \frac1N\sum_i\E\gamma_i(s,X_{t-s})
 \ge\frac m{NM^2V_d(R(s))}[1-p-2q(s)]_+^2.
 \tag{5.10}\label{eq:energy-lower}
\]
The integrand in \eqref{eq:transport-volume} is measurable: each ball volume is a finite sum of indicators of $d(v,w)\le R(s)$, and $V_d$ is their finite maximum. Since $q(s)\ge0$ and $V_d(R(s))\ge1$, the integrand is bounded. Integrate over $[0,T]$. By \cref{lem:variance-budget} and $|f_i|\le1$, the averaged integral on the left is at most one. This proves \eqref{eq:transport-volume}; the numerical specialization follows from $1-1/8-2/8\ge1/2$.
\end{proof}

Structurally, \eqref{eq:transport-volume} parallels the classical on-diagonal heat-kernel lower bound obtained by combining a confinement probability with Cauchy--Schwarz over a ball \cite{CoulhonGrigoryan1997,Barlow2017}; integrating over time produces the Green-function scale $\int ds/V_d(R(s))$. Here no reversible walk or invariant law is used: the response tag is controlled only through the projection representation. A related martingale and localized-influence mechanism underlies logarithmic fluctuation lower bounds in two-dimensional first-passage percolation \cite{NewmanPiza1995}.

The next elementary bound handles parameter ranges in which the transport scale is below one.

\begin{lemma}[Load-sum phase]\label{lem:phase}
For every rule that adds one ball at each event and every $t\ge1/m$,
\[
 \Pp(\Gap(X_t)\ge1)\ge e^{-1}.
 \tag{5.11}\label{eq:phase-long}
\]
\end{lemma}

\begin{proof}
Condition on the process up to time $t-1/m$. The number of events in the final interval is Poisson with mean one and is independent of the past, so the events of zero and one arrival each have conditional probability $e^{-1}$. The phase $\sum_vX(v)\pmod N$ increases by one at every event, whereas a flat profile has phase zero. Given the past, at most one of the zero-event and one-event outcomes can therefore be flat. The terminal profile is nonflat with conditional probability at least $e^{-1}$, and averaging over the past proves the claim.
\end{proof}

A useful consequence of \cref{thm:transport-volume} is obtained by setting
\[
 I=\frac mN\int_0^T\frac{ds}{V_d(R(s))}.
 \tag{5.12}\label{eq:I-def}
\]

\begin{corollary}[From transport volume to a gap]\label{cor:transport-gap}
Assume $q(s)\le1/8$ for $0\le s\le T$, and let $t\ge\max\{T,1/m\}$. Then
\[
 \E\Gap(X_t)\ge\frac{\sqrt I}{32},
 \qquad
 \Pp\!\left(\Gap(X_t)\ge\frac{\sqrt I}{32}\right)\ge\frac18.
 \tag{5.13}\label{eq:I-gap}
\]
\end{corollary}

\begin{proof}
Let $Q=\sqrt I/2$. If $0<b=\E\Gap(X_t)<\infty$, Markov's inequality with $M=8b+1$ gives $p\le1/8$; if $b=0$, take $M=1$. Equation \eqref{eq:transport-volume-simple} then gives $M\ge Q$, so $b\ge(Q-1)/8$. For $Q\ge2$ this is at least $Q/16$, while for $Q<2$ the expectation bound follows from \cref{lem:phase}. If the expectation is infinite, there is nothing to prove.

For the probability estimate, suppose first that $Q\ge4$ and take $M=Q/2$. If $p\le1/8$, then \eqref{eq:transport-volume-simple} would give $M\ge Q$, a contradiction. Thus
\[
 \Pp(\Gap(X_t)>Q/2-1)>1/8,
\]
and $Q/2-1\ge Q/4$. If $Q<4$, \cref{lem:phase} gives the claim because $Q/16<1$.
\end{proof}

\section{The cycle}\label{sec:cycle}

We first combine diffusive displacement with a fixed comparison radius.

\begin{lemma}[Cycle quantile bound]\label{lem:cycle-quantile}
Let $24\le r\le\lfloor n/2\rfloor$, put $R=\lfloor r/3\rfloor$, and assume
\[
 t\ge\frac{R^2}{64\pi^2}.
\]
If $M\ge1$ and
\[
 \Pp(\Gap(X_t)>M-1)\le\frac18,
\]
then
\[
 M\ge\frac{\sqrt r}{64\pi}.
 \tag{6.1}\label{eq:cycle-quantile}
\]
\end{lemma}

\begin{proof}
Take $\psi(i)=i+r\pmod n$, use the cycle metric, and choose the constant radius $R$. By \cref{prop:cycle-diffusion}, for
\[
 0\le s\le T:=\frac{R^2}{64\pi^2},
\]
we have
\[
 q(s)\le\frac{2\pi^2s+2}{R^2}
 \le\frac1{16},
\]
since $R\ge8$. The comparison balls have size $2R+1$, and the two endpoints are separated by at least $3R$. Equation \eqref{eq:transport-volume-simple} gives
\[
 M^2\ge\frac{T}{4(2R+1)}.
\]
Since $R\ge r/4$ and $2R+1\le r$,
\[
 M^2\ge\frac r{4096\pi^2}.
\]
\end{proof}

\begin{proof}[Proof of \cref{thm:intro-cycle}]
Set
\[
 \ell=\min\{n,\sqrt t\},
 \qquad
 \alpha=\sqrt\ell,
 \qquad
 r=\left\lfloor\min\{n/2,\sqrt t\}\right\rfloor.
\]
If $r\ge24$, then $r\ge\ell/3$, and the time condition in \cref{lem:cycle-quantile} holds. Let $b=\E\Gap(X_t)$. If $0<b<\infty$, Markov's inequality with $M=8b+1$ gives
\[
 \Pp(\Gap(X_t)>M-1)\le\frac18.
\]
If $b=0$, the same conclusion holds with $M=1$.
Hence
\[
 b\ge\frac{\alpha}{512\pi\sqrt3}-\frac18.
 \tag{6.2}\label{eq:cycle-expect-pre}
\]
Let $A=(512\pi\sqrt3)^{-1}$. If $A\alpha\ge1/4$, the right side of \eqref{eq:cycle-expect-pre} is at least $A\alpha/2$. If $A\alpha<1/4$, \cref{lem:phase} gives the stronger bound $b\ge e^{-1}$. The same phase bound covers $r<24$. This proves the expectation statement with $c=A/2$; infinite expectation is immediate.

For the probability estimate, suppose first that $r\ge(256\pi)^2$ and set
\[
 M=\frac{\sqrt r}{128\pi}.
\]
If the upper tail at $M-1$ were at most $1/8$, \cref{lem:cycle-quantile} would force $M\ge\sqrt r/(64\pi)$, a contradiction. Therefore
\[
 \Pp(\Gap(X_t)>M-1)>\frac18.
\]
Since $M\ge2$ and $r\ge\ell/3$,
\[
 M-1\ge\frac{\sqrt r}{256\pi}
 \ge\frac{\alpha}{256\pi\sqrt3}
 \ge c\alpha.
\]
If $r<(256\pi)^2$, then $c\alpha<1$, and \cref{lem:phase} proves the same probability bound.

For an invariant normalized law, start the process with that law and take $t\ge n^2$.
\end{proof}

\subsection{A deterministic number of allocations}

Recall that $\mathsf P$ is the one-event transition operator. The continuous-time semigroup is
\[
 P_t=\exp(mt(\mathsf P-I)).
\]
The discrete martingale bracket contains a mean-increment correction that has no continuous-time analogue.

\begin{lemma}[Discrete variance correction]\label{lem:discrete-correction}
Fix a real number $M>0$, a vertex $i$, and a vertex $\psi(i)\ne i$. Let
\[
 f(x)=\clip_{[-1,1]}\!\left(\frac{x_i-x_{\psi(i)}}{M}\right),
\]
and define
\[
 \gamma_h(x)=\sum_v\lambda_v(x)|D_v\mathsf P^hf(x)|^2.
\]
Then
\[
 \mathsf P[(\mathsf P^hf)^2](x)-(\mathsf P^{h+1}f(x))^2
 =\frac{\gamma_h(x)}m-
 |(\mathsf P-I)\mathsf P^hf(x)|^2.
 \tag{6.3}\label{eq:discrete-var}
\]
If $G$ has maximum degree $\Delta$, then
\[
 |(\mathsf P-I)\mathsf P^hf(x)|
 \le\frac{2\Delta}{mM}.
 \tag{6.4}\label{eq:mean-increment}
\]
\end{lemma}

\begin{proof}
Write $g=\mathsf P^hf$. Since a one-event update allocates to $v$ with probability $\lambda_v(x)/m$,
\[
 \mathsf Pg(x)=g(x)+\frac1m\sum_v\lambda_v(x)D_vg(x).
\]
Expanding $\mathsf P(g^2)- (\mathsf Pg)^2$ cancels the constant and linear terms and gives \eqref{eq:discrete-var}.

For \eqref{eq:mean-increment}, the discrete Palm link and \eqref{eq:derivative-tag-discrete} give
\[
 m(\mathsf P-I)\mathsf P^hf(x)
 =\E_x\sum_v\lambda_v(X^{[h]})D_vf(X^{[h]}).
\]
Only $i$ and $\psi(i)$ contribute. Each corresponding allocation rate is at most $\Delta$, and each derivative has absolute value at most $1/M$.
\end{proof}

\begin{lemma}[Discrete transport-volume inequality]\label{lem:discrete-transport}
Let $\Delta$ be the maximum degree of $G$. Fix an integer $k\ge0$, a permutation $\psi$ of $V$, and an integer $H\le k+1$. For each $0\le h<H$, let $R_h>0$ satisfy
\[
 d(i,\psi(i))\ge2R_h
 \qquad(i\in V),
\]
and suppose that, uniformly in the starting profile, the Palm tag after exactly $h$ allocations obeys
\[
 \Pp(d(J_h,J_0)\ge R_h)\le q_h.
\]
Then, for every $M\ge1$, with
\[
 p=\Pp(\Gap(X^{[k]})>M-1),
\]
one has
\[
 M^2\ge\sum_{h=0}^{H-1}
 \left[
 \frac{[1-p-2q_h]_+^2}{N V_d(R_h)}-
 \frac{4\Delta^2}{m^2}
 \right]_+.
 \tag{6.5}\label{eq:discrete-transport}
\]
If $p,q_h\le1/8$ and
\[
 V_d(R_h)\le\frac{m^2}{32N\Delta^2},
 \tag{6.6}\label{eq:absorb-discrete}
\]
then the $h$th summand in \eqref{eq:discrete-transport} is at least $1/(8N V_d(R_h))$.
\end{lemma}

\begin{proof}
For each $i\in V$, let
\[
 f_i(x)=\clip_{[-1,1]}\!\left(\frac{x_i-x_{\psi(i)}}M\right).
\]
Use the Doob martingale with terminal variable $f_i(X^{[k+1]})$. For $0\le h\le k$, define
\[
 c_{i,h}(x)
 =\mathsf P[(\mathsf P^hf_i)^2](x)
   -(\mathsf P^{h+1}f_i(x))^2.
\]
The increment at time $k-h$, conditioned on $X^{[k-h]}$, has variance
$c_{i,h}(X^{[k-h]})$. Set
\[
 C_h=\frac1N\sum_i\E c_{i,h}(X^{[k-h]}).
\]
Each $C_h$ is nonnegative because it is an averaged conditional variance. The indices $h=0,\ldots,H-1$ correspond to distinct martingale increments, so the martingale isometry and $|f_i|\le1$ give
\[
 \sum_{h=0}^{H-1}C_h
 \le\frac1N\sum_i\Var(f_i(X^{[k+1]}))
 \le1.
 \tag{6.7}\label{eq:discrete-bracket-budget}
\]
The terminal time $k+1$ aligns the response correctly: the lag-$h$ derivative is evaluated at $X^{[k-h]}$, and its unperturbed $h$-step Palm base terminates at $X^{[k]}$, where the gap tail $p$ is measured.

Write
\[
 \gamma_{i,h}(x)=\sum_v\lambda_v(x)
 |D_v\mathsf P^hf_i(x)|^2.
\]
The event-count derivative identity and Palm link from \cref{prop:palm} give, with $A_i=B(i,R_h)$,
\[
 S_h:=\frac1N\sum_i\E\sum_{v\in A_i}
 \lambda_v(X^{[k-h]})D_v\mathsf P^hf_i(X^{[k-h]})
 =\frac mN\sum_i\E\bigl[
 \one_{\{J_0\in A_i\}}D_{J_h}f_i(X^{[k]})
 \bigr].
\]
We now carry out the translated-test count. Without the restriction
$J_0\in A_i$, the positive terminal contribution is
\[
 \frac mN\sum_i\E\!
 \left[\one_{\{J_h=i\}}D_if_i(X^{[k]})\right].
\]
By \eqref{eq:conditional-palm} and \eqref{eq:rates}, this expression equals
\[
 \frac1N\E\sum_i\lambda_i(X^{[k]})D_if_i(X^{[k]})
 \ge \frac m{NM}(1-p).
\]
Indeed, on $\{\Gap(X^{[k]})\le M-1\}$ one has $D_if_i(X^{[k]})=1/M$ and $\sum_i\lambda_i(X^{[k]})=m$.
Restricting to $J_0\in A_i$ can remove positive credit only on
\[
 \{J_h=i,\ J_0\notin A_i\},
\]
which implies $d(J_h,J_0)\ge R_h$.  The only negative terminal derivative is
at $\psi(i)$; its contribution with $J_0\in A_i$ is supported on
\[
 \{J_h=\psi(i),\ J_0\in A_i\},
\]
and this event also implies $d(J_h,J_0)\ge R_h$ because
$d(i,\psi(i))\ge2R_h$. Since both $i\mapsto i$ and
$i\mapsto\psi(i)$ are permutations, each exceptional family has total cost
at most $mq_h/(NM)$. Therefore
\[
 S_h\ge\frac m{NM}(1-p-2q_h).
\]
Cauchy--Schwarz on the restricted Palm measure, whose total mass is at most $mV_d(R_h)/N$, yields
\[
 \frac1N\sum_i\E\gamma_{i,h}(X^{[k-h]})
 \ge
 \frac m{NM^2V_d(R_h)}[1-p-2q_h]_+^2.
 \tag{6.8}\label{eq:discrete-energy-lower}
\]
By \cref{lem:discrete-correction},
\[
 C_h
 =\frac1{mN}\sum_i\E\gamma_{i,h}(X^{[k-h]})
 -\frac1N\sum_i\E
 |(\mathsf P-I)\mathsf P^hf_i(X^{[k-h]})|^2.
\]
Using \eqref{eq:mean-increment} and \eqref{eq:discrete-energy-lower},
\[
 C_h\ge\frac1{M^2}
 \left(
 \frac{[1-p-2q_h]_+^2}{NV_d(R_h)}
 -\frac{4\Delta^2}{m^2}
 \right).
\]
Since $C_h\ge0$, it is at least the positive part of the right-hand side. Summing over $h$ and applying \eqref{eq:discrete-bracket-budget} proves \eqref{eq:discrete-transport}. Under the final two hypotheses, the first term in parentheses is at least $1/(4NV_d(R_h))$, while \eqref{eq:absorb-discrete} makes the subtraction at most $1/(8NV_d(R_h))$.
\end{proof}

\begin{lemma}[One-step nonflatness]\label{lem:discrete-small}
On $C_n$, for every initial law and every $k\ge1$,
\[
 \Pp(\Gap(X^{[k]})\ge1)\ge1-\frac{2}{n}\ge\frac13.
 \tag{6.9}\label{eq:one-step-nonflat}
\]
\end{lemma}

\begin{proof}
A nonflat profile can become flat in one update only when exactly one vertex is one unit below all the others, and the ball is allocated to that unique deficient vertex. Its selection probability is at most $2/n$. A flat profile has no flat one-step successor.
\end{proof}

\begin{proof}[Proof of \cref{thm:intro-discrete}]
Set
\[
 \ell=\min\{n,\sqrt{k/n}\},
 \qquad
 \alpha=\sqrt\ell,
 \qquad
 R=\left\lfloor\min\{n/512,\sqrt{k/n}\}\right\rfloor.
\]
Suppose first that $R\ge8$. Let
\[
 H=\left\lfloor\frac{nR^2}{64\pi^2}\right\rfloor,
 \qquad
 \psi(i)=i+\lfloor n/2\rfloor\pmod n.
\]
Then $1\le H\le k$, and the condition $R\le n/512$ implies $2R+1\le n/128$. The event-count version of \cref{prop:cycle-diffusion} gives
\[
 q_h\le\frac{2\pi^2h/n+2}{R^2}\le\frac1{16}
 \qquad(0\le h<H).
\]
If $p\le1/8$, \cref{lem:discrete-transport}, with $m=N=n$ and $\Delta=2$, gives
\[
 M^2\ge\frac{H}{8n(2R+1)}.
\]
Since $R\ge8$ and $R\le n/512$, the number $nR^2/(64\pi^2)$ exceeds two. Hence
\[
 H\ge\frac{nR^2}{128\pi^2},
\]
and therefore
\[
 M^2\ge\frac{R}{3072\pi^2}.
\]
The definition of $R$ and the assumption $R\ge8$ imply $R\ge\ell/1024$. Thus
\[
 M^2\ge\frac{\ell}{3\cdot2^{20}\pi^2}.
 \tag{6.10}\label{eq:discrete-quantile}
\]
Let
\[
 Q=\frac{\alpha}{1024\pi\sqrt3}.
\]
Equation \eqref{eq:discrete-quantile} says that $p\le1/8$ forces $M\ge Q$.

For $0<b=\E\Gap(X^{[k]})<\infty$, use $M=8b+1$; if $b=0$, take $M=1$. When $Q\ge2$, this gives $b\ge(Q-1)/8\ge Q/16$, proving the expectation estimate with $c'=1/(16384\pi\sqrt3)$.

For the probability estimate, when $Q\ge4$ take $M=Q/2$. The assumption $p\le1/8$ would contradict \eqref{eq:discrete-quantile}; hence
\[
 \Pp(\Gap(X^{[k]})>Q/2-1)>\frac18,
\]
and $Q/2-1\ge Q/4\ge c'\alpha$. In every complementary parameter range ($Q<4$ or $R<8$), the claimed threshold is below $1/3$, so \cref{lem:discrete-small} proves both assertions.
\end{proof}

\section{Other graphs}\label{sec:graphs}

The projection argument gives displacement bounds in any Hilbert embedding. This is the geometric input to \cref{thm:transport-volume}; the theorem itself does not assert graph-distance diffusion on arbitrary graphs.

\begin{lemma}[Hilbert displacement]\label{lem:hilbert-displacement}
Let an endpoint-local monotone rule run on $G$, initialize the tag according to \eqref{eq:palm-init}, and let $F:V\to \mathcal H$ satisfy
\[
 \norm{F(u)-F(v)}\le \eta
 \qquad(uv\in E).
\]
If $\Delta$ is the maximum degree, then, uniformly over the starting profile,
\[
 \E\norm{F(J_s)-F(J_0)}^2
 \le2\eta^2\Delta(\Delta+1)s+2\eta^2.
 \tag{7.1}\label{eq:hilbert-diffusion}
\]
After exactly $h$ events, the same estimate holds with $s$ replaced by $h/m$.
\end{lemma}

\begin{proof}
Assign to a marked edge $e=uv$ its midpoint
\[
 g(e)=\frac{F(u)+F(v)}2.
\]
Every selection cell at $v$ lies in the Hilbert ball of radius $\eta/2$ around $F(v)$, so its conditional variance is at most $\eta^2/4$. At any update $j$, if the selected vertex is $w$, the affected cells are supported on edge fibers incident to $N[w]$. There are at most $\Delta(\Delta+1)$ such fibers, and each has $\rho$-mass $1/m$. Hence
\[
 \norm{(I-Q_j)g}_{L^2(\rho;\mathcal H)}^2
 \le\frac{\eta^2\Delta(\Delta+1)}{4m}.
\]
By \cref{lem:projections}, the expected squared midpoint displacement after $h$ events is at most $\eta^2\Delta(\Delta+1)h/m$. The initial and final endpoint-to-midpoint errors have total norm at most $\eta$, so
\[
 \norm{F(J_h)-F(J_0)}^2
 \le2\norm{g(B_h)-g(B_0)}^2+2\eta^2.
\]
Average over the Poisson event count for \eqref{eq:hilbert-diffusion}.
\end{proof}

Suppose $D>0$ and a pseudometric $d$ satisfies
\[
 d(u,v)\le D\norm{F(u)-F(v)},
 \qquad
 \norm{F(u)-F(v)}\le1\quad(uv\in E).
 \tag{7.2}\label{eq:embedding-distortion}
\]
Then Markov's inequality and \cref{lem:hilbert-displacement} give the tail envelope
\[
 q(s)\le
 \min\left\{1,
 \frac{D^2(2\Delta(\Delta+1)s+2)}{R(s)^2}
 \right\}.
 \tag{7.3}\label{eq:embedding-tail}
\]
For a constant $R\ge8D$, one may take
\[
 T=\frac{R^2}{64D^2\Delta(\Delta+1)},
\]
which makes $q(s)\le1/16$ on $[0,T]$. If a permutation separated by $2R$ exists, then for every $t\ge\max\{T,1/m\}$, \cref{cor:transport-gap} gives the lower scale
\[
 \frac{R}{256D\sqrt{\Delta(\Delta+1)}}
 \sqrt{\frac m{N V_d(R)}}.
 \tag{7.4}\label{eq:single-scale-graph}
\]

For graph distance, with $d(u,v)=\infty$ when $u$ and $v$ lie in different components, and integer $R\ge1$, a separated permutation exists whenever
\[
 \max_v|B(v,2R-1)|\le N/2.
 \tag{7.5}\label{eq:hall-condition}
\]
Indeed, form the bipartite graph whose left and right copies of $V$ are joined when their graph distance is at least $2R$. Every vertex has degree at least $N/2$. If a left set $S$ has size at most $N/2$, the neighborhood of any member already has size at least $|S|$. If $|S|>N/2$, every right vertex has a neighbor in $S$, because it has at most $N/2$ nonneighbors. Thus Hall's condition holds; Hall's marriage theorem supplies the required permutation.

\subsection{Finite-width cylinders}

\begin{corollary}[Cycles of finite width]\label{cor:cylinder}
Let $H$ be a finite simple graph with at least one vertex, let $L\ge3$, write $w=|V(H)|$, and put $\Delta=2+\Delta(H)$. On $C_L\square H$, under every endpoint-local monotone rule and every initial law, for all $t\ge L^2$,
\[
 \E\Gap(X_t)\ge c_\Delta\sqrt{L/w},
\]
and
\[
 \Pp\!\left(\Gap(X_t)\ge c_\Delta\sqrt{L/w}\right)\ge\frac18,
\]
where
\[
 c_\Delta=\frac1{2048\pi\sqrt{\Delta(\Delta+1)}}.
\]
The graph $H$ may be disconnected. The same conclusions hold under every invariant normalized law.
\end{corollary}

\begin{proof}
Use the pseudometric that records only cyclic distance in the $C_L$ coordinate. The map
\[
 F(u,h)=\frac{e^{2\pi iu/L}}{2\sin(\pi/L)}
\]
is edge-Lipschitz with constant one and satisfies
\[
 d((u,h),(v,h'))\le\frac\pi2
 \norm{F(u,h)-F(v,h')}.
\]
A ball of radius $R<L/2$ has at most $w(2R+1)$ vertices. Pair $(u,h)$ with $(u+\lfloor L/2\rfloor,h)$. The average total allocation rate is
\[
 \frac mN=1+\frac{|E(H)|}{w}\ge1.
\]

For $L\ge120$, choose $R=\lfloor L/6\rfloor$. Then $R\ge L/7$, $R\ge4\pi$, and $2R+1\le L/2$. Substitution into \eqref{eq:single-scale-graph} gives at least
\[
 \frac{\sqrt2}{896\pi\sqrt{\Delta(\Delta+1)}}\sqrt{L/w},
\]
which is larger than the stated constant. The required horizon is at most $L^2$. For $3\le L<120$, the stated threshold is below $e^{-1}$, so \cref{lem:phase} applies. Starting from an invariant normalized law gives the stationary assertion.
\end{proof}

\subsection{Rectangular tori}

\begin{theorem}[Rectangular two-dimensional tori]\label{thm:rectangular-torus}
Let $3\le K\le L$. On $C_L\square C_K$, under every endpoint-local monotone rule and every initial law, for all $t\ge L^2$,
\[
 \E\Gap(X_t)\ge c_2\sqrt{L/K+\log K},
\]
and
\[
 \Pp\!\left(\Gap(X_t)\ge c_2\sqrt{L/K+\log K}\right)\ge\frac18,
\]
where
\[
 c_2=\frac1{1280\pi\sqrt{42}}.
\]
The same conclusions hold under every invariant normalized law.
\end{theorem}

\begin{proof}
Embed each cyclic coordinate in its edge-normalized circle and take their Hilbert direct sum. The resulting map is edge-Lipschitz with constant one. If $d$ is graph distance, then
\[
 d(x,y)\le D\norm{F(x)-F(y)},
 \qquad D=\frac\pi{\sqrt2}.
\]
Indeed, in each coordinate the normalized chord is at least $2/\pi$ times cyclic distance, and $(a+b)^2\le2(a^2+b^2)$.

Here $\Delta=4$ and $m/N=2$. By \cref{lem:hilbert-displacement},
\[
 \E d(J_s,J_0)^2\le C(s+1),
 \qquad C=21\pi^2.
\]
Set
\[
 \kappa=4\sqrt C=4\pi\sqrt{21},
 \qquad
 R(s)=\lceil \kappa\sqrt{s+1}\rceil,
 \qquad
 T=\left(\frac{L}{16\kappa}\right)^2-1.
\]
When $L\ge16\kappa$, the radii are at most $L/8$ on $[0,T]$, and the displacement tail is at most $1/16$. Pair every vertex with its shift by $\lfloor L/2\rfloor$ in the long coordinate. A metric ball satisfies
\[
 V_d(R)\le(2R+1)\min\{K,2R+1\}.
 \tag{7.6}\label{eq:torus-volume}
\]
With $r=\kappa\sqrt{s+1}$, we have $2R(s)+1\le2r+3\le5r$, so
\begin{align}
 I
 &=2\int_0^T\frac{ds}{V_d(R(s))} \notag\\
 &\ge\frac4{25\kappa^2}
 \int_\kappa^{L/16}\frac{dr}{\min\{K,r\}}.
 \tag{7.7}\label{eq:torus-integral}
\end{align}

Write
\[
 Z=\frac LK+\log K.
\]
Since $K\le L$,
\[
 \int_1^L\frac{dr}{\min\{K,r\}}=Z-1.
\]
The integrand is nonincreasing. Rescaling the interval by $16$ and removing the bounded initial interval gives
\[
 \int_\kappa^{L/16}\frac{dr}{\min\{K,r\}}
 \ge\frac{Z}{16}-1-\kappa.
\]
Consequently
\[
 I\ge AZ-B,
 \qquad
 A=\frac1{100\kappa^2},
 \qquad
 B=\frac{4(1+\kappa)}{25\kappa^2}.
\]
If $Z\ge2B/A=32(1+\kappa)$, \cref{cor:transport-gap} gives the stated lower bounds because
\[
 \frac{\sqrt{A/2}}{32}
 =\frac1{320\sqrt2\,\kappa}
 =\frac1{1280\pi\sqrt{42}}.
\]
If $Z<32(1+\kappa)$, then $c_2\sqrt{Z}<e^{-1}$, and \cref{lem:phase} applies. The remaining case $L<16\kappa$ also satisfies
\[
 Z\le16\kappa/3+\log(16\kappa)<32(1+\kappa),
\]
so it is covered by the same phase argument. Starting from an invariant normalized law gives the stationary assertion.
\end{proof}

\begin{remark}[One family of terminal tests]\label{rem:one-family}
The logarithm in \cref{thm:rectangular-torus} is not obtained by summing separate variance bounds for different observable families. The permutation and the clipped terminal contrasts remain fixed. Only the protected radius $R(s)$ varies with the response lag in the single integral \eqref{eq:transport-volume}.
\end{remark}

For $K$ bounded, \cref{thm:rectangular-torus} recovers square-root growth in the long direction. Its logarithmic term, however, is weaker than a bound available for every allocation strategy. Bansal and Feldheim proved a strategy-independent $\Omega(\log n)$ lower bound on every graph \cite{BansalFeldheim}. We include a short self-contained version with the explicit time window and probability needed here, and then combine it with \cref{thm:rectangular-torus}.

\begin{proposition}[A logarithmic gap for every strategy]\label{prop:universal-log}
Let $G$ be a $\Delta$-regular graph on $N$ vertices, and let balls be allocated by an arbitrary strategy: each arriving ball is placed at one endpoint of the arriving edge, and the choice may depend on the entire history and on additional randomness. Put $h_N=\lfloor (N/16)\log N\rfloor$. If
\[
 N\ge\max\{1000,\ (9(\Delta+1))^{4/3}\},
\]
then for every initial law independent of future arrivals and every $k\ge h_N$,
\[
 \Pp\!\left(\Gap(X^{[k]})\ge \tfrac1{64}\log N\right)\ge\frac12.
 \tag{7.8}\label{eq:universal-log}
\]
In physical time, with every edge ringing at rate one, $\Pp(\Gap(X_t)\ge\frac1{64}\log N)\ge\frac7{16}$ for every $t\ge(\log N)/(4\Delta)$.
\end{proposition}

\begin{proof}
Each allocation selects a uniformly random edge, so a fixed vertex is incident to it with probability $\Delta/m=2/N$, whatever the strategy. Condition on the history up to allocation $k-h_N$, let $x$ be the load vector at that time, and write $a=N^{-1}\sum_vx_v$ and $\tau=h_N/N$. After $h_N$ further allocations the average load is $a+\tau$. Since $N\ge1000$, we have $\tau\ge(\log N)/16-1/N\ge(\log N)/32$.

Suppose first that some vertex $v$ has $x_v\le a-4\tau$. Its number $\nu_v$ of incident arrivals in the window has mean $2\tau$, so Markov's inequality gives $\Pp(\nu_v\le4\tau)\ge1/2$. On this event the final load of $v$ is at most $a$, while the final average is $a+\tau$; hence the gap is at least $\tau$.

Otherwise every vertex has $x_v>a-4\tau$. Then at least $N/9$ vertices satisfy $x_v\le a+\tau/2$: if a set $S$ of fewer than $N/9$ vertices contained all of them, then
\[
 \sum_vx_v>|S|(a-4\tau)+(N-|S|)(a+\tau/2)=Na+\tau\Bigl(\frac N2-\frac92|S|\Bigr)>Na,
\]
a contradiction. Choosing greedily, these vertices contain an independent set $\mathcal I$ with $|\mathcal I|\ge N/(9(\Delta+1))$; the edge sets incident to distinct vertices of $\mathcal I$ are disjoint. Let $Y$ be the number of vertices of $\mathcal I$ with no incident arrival in the window. Since $\log(1-x)\ge-2x$ for $0\le x\le1/2$,
\[
 \E Y=|\mathcal I|\Bigl(1-\frac2N\Bigr)^{h_N}\ge |\mathcal I|e^{-4\tau}\ge\frac{N^{3/4}}{9(\Delta+1)}\ge1.
\]
For distinct $u,w\in\mathcal I$, the probability that neither has an incident arrival is $(1-4/N)^{h_N}\le(1-2/N)^{2h_N}$, so $\Var Y\le\E Y$. Hence
\[
 \Pp(Y>0)\ge\frac{(\E Y)^2}{\E Y^2}\ge\frac{\E Y}{1+\E Y}\ge\frac12.
\]
An untouched vertex of $\mathcal I$ keeps its load, which is at most $a+\tau/2$, so the gap is at least $\tau/2\ge(\log N)/64$. Both cases give \eqref{eq:universal-log}, and the bound is uniform in the conditioned history.

In physical time, the number $N_t$ of events up to time $t$ is Poisson with mean $\mu=mt=\Delta Nt/2\ge(N/8)\log N\ge2h_N$. Since $\mu\ge32$, Chebyshev's inequality gives $\Pp(N_t\ge\mu/2)\ge7/8$. Conditional on the event times, the arriving edges are independent and uniform, so the discrete argument applies to the last $h_N$ events on $\{N_t\ge h_N\}$. This gives probability at least $(7/8)(1/2)=7/16$.
\end{proof}

\begin{corollary}[Rectangular tori: combined bound]\label{cor:torus-combined}
Let $3\le K\le L$ and $N=LK$. On $C_L\square C_K$, under every endpoint-local monotone rule and every initial law, for all $t\ge L^2$,
\[
 \E\Gap(X_t)\ge c_3\bigl(\sqrt{L/K}+\log N\bigr),
 \qquad
 \Pp\!\left(\Gap(X_t)\ge c_3\bigl(\sqrt{L/K}+\log N\bigr)\right)\ge\frac18,
\]
where $c_3=c_2/2$. The same conclusions hold under every invariant normalized law.
\end{corollary}

\begin{proof}
The torus is $4$-regular, and $(9\cdot5)^{4/3}<1000$. Since $N\le L^2$, we have $(\log N)/16\le L^2\le t$.

Suppose $N\ge1000$, and put $a=c_2\sqrt{L/K}$ and $b=(\log N)/64$. By \cref{thm:rectangular-torus}, $\Pp(\Gap(X_t)\ge a)\ge1/8$, and by \cref{prop:universal-log}, $\Pp(\Gap(X_t)\ge b)\ge7/16$. Since $c_2<1/64$,
\[
 \max\{a,b\}\ge\frac{a+b}2\ge\frac{c_2}2\bigl(\sqrt{L/K}+\log N\bigr),
\]
and the event $\{\Gap(X_t)\ge\max\{a,b\}\}$ has probability at least $1/8$. For the expectation, $\E\Gap(X_t)\ge\max\{c_2\sqrt{L/K},\,\tfrac{7}{1024}\log N\}$, and $c_2<7/1024$ gives the same bound.

If $N<1000$, then $L/K<1000$ and $c_3(\sqrt{L/K}+\log N)<e^{-1}$, so \cref{lem:phase} gives both statements. Starting from an invariant normalized law gives the stationary assertion.
\end{proof}

On the square torus $C_L\square C_L$, the strategy-independent scale $\log N$ dominates the contribution of the transport-volume inequality. The new content of \cref{cor:torus-combined} for endpoint-local rules is the aspect-ratio term $\sqrt{L/K}$.

\section{A nearly matching rule}\label{sec:smoothed}

The lower bounds above concern every endpoint-local monotone rule. This section shows that the cycle bound is sharp for the class up to a logarithmic factor. The rule used here has drift that is linear in the load gradient as long as neighboring loads stay within a cutoff, so the load process is a stopped linear recursion driven by bounded martingale noise. The smoothed threshold rule is a graphical form of two-choice allocation with probabilistic comparison noise, a setting studied on complete graphs by Los and Sauerwald \cite{LosSauerwald2023}; the graphical-removal analysis of Olesker-Taylor, Sauerwald, and Zanetti also permits noise in how load is allocated \cite{OleskerTaylorSauerwaldZanetti2026}.

The Laplacian-pseudoinverse control below is also close in spirit to effective-resistance analyses of dynamic averaging. Berenbrink, Hintze, Hosseinpour, Kaaser, and Rau use effective-resistance methods for dynamic averaging on arbitrary graphs \cite{BerenbrinkEtAl2023}, while Kraizberg's later pairwise concentration bounds are governed by effective resistance \cite{Kraizberg2026}.

Throughout this section $G$ is a connected $d$-regular graph on $N\ge3$ vertices with $m=Nd/2$ edges. Let $\mathcal L$ be its graph Laplacian and $\mathcal L^+$ the pseudoinverse, that is, the inverse of $\mathcal L$ on mean-zero vectors and zero on constants. Put
\[
 \mathsf R_G=\max_{v\in V}(\mathcal L^+)_{vv}.
\]
For a cutoff $\theta>0$, the \emph{smoothed threshold rule} uses, on every edge,
\[
 p_e(z)=\min\Bigl\{1,\ \max\Bigl\{0,\ \frac12-\frac z{2\theta}\Bigr\}\Bigr\}.
\]
This is an endpoint-local monotone rule. Since $1-p_e(z)=p_e(-z)$, it does not depend on the orientation of the edge. Equivalently, an independent threshold uniformly distributed on $[-\theta,\theta]$ is drawn at each event.

\begin{lemma}[Bounded-increment Bernstein inequality]\label{lem:bernstein}
Let $b\ge0$, and let $(\xi_j)$ be a finite sequence of martingale differences with $|\xi_j|\le b$ and $\sum_j\E(\xi_j^2\mid\mathcal F_{j-1})\le\Sigma$ almost surely, where $\Sigma>0$ is deterministic. Then for every $s>0$,
\[
 \Pp\Bigl(\Bigl|\sum_j\xi_j\Bigr|\ge\sqrt{2\Sigma s}+\frac{2b}3s\Bigr)\le2e^{-s}.
 \tag{8.1}\label{eq:bernstein}
\]
\end{lemma}

\begin{proof}
If $b=0$, all increments vanish and the conclusion is immediate. Assume $b>0$. For $|z|<3$, the bound $k!\ge2\cdot3^{k-2}$ for $k\ge2$ gives $e^z\le1+z+z^2/(2(1-|z|/3))$. For $0<\chi<3/b$ it follows that
\[
 \E\bigl(e^{\chi\xi_j}\mid\mathcal F_{j-1}\bigr)
 \le\exp\bigl(c_\chi\E(\xi_j^2\mid\mathcal F_{j-1})\bigr),
 \qquad c_\chi=\frac{\chi^2}{2(1-\chi b/3)}.
\]
Thus $\exp(\chi\sum_{i\le j}\xi_i-c_\chi\sum_{i\le j}\E(\xi_i^2\mid\mathcal F_{i-1}))$ is a supermartingale, and $\E e^{\chi\sum_j\xi_j}\le e^{c_\chi\Sigma}$. For $a>0$ choose $\chi=a/(\Sigma+ba/3)$; then $\Pp(\sum_j\xi_j\ge a)\le\exp(-a^2/(2(\Sigma+ba/3)))$. Apply this to $-\xi_j$ as well. Finally, if $a\ge\sqrt{2\Sigma s}+2bs/3$, then $a^2\ge2s(\Sigma+ba/3)$, because $a^2-\frac{2bs}3a-2\Sigma s$ is increasing for $a\ge bs/3$ and equals $\frac{2bs}3\sqrt{2\Sigma s}\ge0$ at $a=\sqrt{2\Sigma s}+2bs/3$.
\end{proof}

\begin{theorem}[Smoothed threshold rule]\label{thm:smoothed}
Let $\zeta\ge1$, put $\beta=2\zeta+2$, and run the smoothed threshold rule with cutoff
\[
 \theta=(4d+3)\beta\log N
\]
on $G$, starting from a constant profile. Then for every $1\le k\le N^\zeta$,
\[
 \E\Gap(X^{[k]})
 \le 2\beta\log N\Bigl(2\sqrt{d(4d+3)\mathsf R_G}+\frac43\Bigr)+2.
 \tag{8.2}\label{eq:smoothed-bound}
\]
\end{theorem}

\begin{proof}
Write $X_k=X^{[k]}$ and $Y_k=X_k-N^{-1}(\sum_vX_k(v))\one$, so that $Y_0=0$ and $\Gap(X_k)=\Gap(Y_k)$. Put $\alpha=(2m\theta)^{-1}$ and $\mathsf{W}=I-\alpha\mathcal L$, and let
\[
 T=\inf\{k\ge0:\ |Y_k(u)-Y_k(v)|>\theta\ \text{for some edge }uv\}.
\]

\emph{Linear drift.} If every edge difference of $x$ is at most $\theta$ in absolute value, no clipping occurs, and the probability that the next ball goes to $v$ is
\[
 q_v(x)=\frac1m\sum_{w\sim v}\Bigl(\frac12-\frac{x_v-x_w}{2\theta}\Bigr)=\frac1N-\alpha(\mathcal Lx)_v.
\]
Let $D_{k+1}=e_{w_{k+1}}-q(X_k)$, where $w_{k+1}$ is the vertex receiving ball $k+1$. Then $(D_k)$ are martingale differences, $\one^{\mathsf T}D_{k+1}=0$, and for every $k<T$,
\[
 Y_{k+1}=\mathsf{W}Y_k+D_{k+1}.
\]
For every vector $g$, $|g^{\mathsf T}D_{k+1}|\le2\norm{g}_\infty$, and since a vertex receives a ball only through an incident edge, $q_v\le d/m=2/N$ at every state, so
\[
 \E\bigl((g^{\mathsf T}D_{k+1})^2\mid\mathcal F_k\bigr)\le\sum_vq_v g_v^2\le\frac2N\norm{g}_2^2.
\]

\emph{Stopped linear process.} Define $Z_0=0$ and $Z_{k+1}=\mathsf{W}Z_k+\one_{\{k<T\}}D_{k+1}$. By induction, $Z_k=Y_k$ for every $k\le T$. Because $\alpha d=1/(N\theta)\le1/2$, the matrix $\mathsf{W}$ is symmetric and stochastic with nonnegative entries, and its eigenvalues on mean-zero vectors are $1-\alpha\omega\in[0,1)$, where $0<\omega\le2d$ runs over the nonzero Laplacian eigenvalues. For a mean-zero vector $\ell$ with $\norm{\ell}_\infty\le1$ and fixed $k$,
\[
 \ell^{\mathsf T}Z_k=\sum_{i=0}^{k-1}\one_{\{i<T\}}\,(\mathsf{W}^{k-1-i}\ell)^{\mathsf T}D_{i+1}
\]
is a sum of martingale differences bounded by $2$. Since $1-(1-\alpha\omega)^2\ge\alpha\omega$, its conditional variances sum to at most
\[
 \frac2N\sum_{h\ge0}\norm{\mathsf{W}^h\ell}_2^2\le\frac2{N\alpha}\,\ell^{\mathsf T}\mathcal L^+\ell
 =2d\theta\,\ell^{\mathsf T}\mathcal L^+\ell.
 \tag{8.3}\label{eq:smoothed-budget}
\]

\emph{Edges.} For an edge $e=uv$ and $\ell_e=e_u-e_v$, we have $\ell_e^{\mathsf T}\mathcal L^+\ell_e\le1$: with $f=\mathcal L^+\ell_e$ and $r=\ell_e^{\mathsf T}\mathcal L^+\ell_e$, the inequality $(\ell_e^{\mathsf T}f)^2\le f^{\mathsf T}\mathcal Lf$ reads $r^2\le r$. Take $s=\beta\log N$, so that $\theta=(4d+3)s$. Since $\sqrt{4d(4d+3)}\le4d+\frac32$,
\[
 \sqrt{2\cdot2d\theta\cdot s}+\frac43s\le\Bigl(4d+\frac32+\frac43\Bigr)s\le\theta.
\]
\Cref{lem:bernstein} gives $\Pp(|\ell_e^{\mathsf T}Z_j|>\theta)\le2N^{-\beta}$ for every $j$ and $e$. If $T\le k$, then $Z_T=Y_T$ has an edge difference exceeding $\theta$. Hence
\[
 \Pp(T\le k)\le2kmN^{-\beta}.
\]

\emph{Vertices.} For $\ell_v=e_v-N^{-1}\one$ we have $\ell_v^{\mathsf T}\mathcal L^+\ell_v=(\mathcal L^+)_{vv}\le\mathsf R_G$ and $\ell_v^{\mathsf T}Z_k=Z_k(v)$. With $\Sigma=2d\theta\mathsf R_G$ and $s=\beta\log N$, \cref{lem:bernstein} gives $\Pp(|Z_k(v)|>R)\le2N^{-\beta}$, where
\[
 R=\sqrt{2\Sigma s}+\frac43s=\beta\log N\Bigl(2\sqrt{d(4d+3)\mathsf R_G}+\frac43\Bigr).
\]

\emph{Conclusion.} Outside the event $\{T\le k\}\cup\{\max_v|Z_k(v)|>R\}$ we have $Y_k=Z_k$ and $\Gap(X_k)\le2R$. Starting from a constant profile, $\Gap(X_k)\le k$ always. Therefore
\[
 \E\Gap(X_k)\le2R+k\bigl(2km+2N\bigr)N^{-\beta}.
\]
For $k\le N^\zeta$ and $m\le N^2/2$, the last term is at most $(dN^{2\zeta+1}+2N^{\zeta+1})N^{-2\zeta-2}\le2$.
\end{proof}

\begin{corollary}[Cycle]\label{cor:smoothed-cycle}
On $C_n$, $n\ge3$, the smoothed threshold rule with $\theta=242\log n$, started from the flat profile, satisfies
\[
 \E\Gap(X^{[k]})\le300\sqrt n\log n\qquad(1\le k\le n^{10}),
\]
and, with every edge ringing at rate one,
\[
 \E\Gap(X_t)\le300\sqrt n\log n+2\qquad(0\le t\le n^4).
\]
Together with \cref{thm:intro-cycle}, this gives, for $n^2\le t\le n^4$,
\[
 c\sqrt n\le\E\Gap(X_t)\le300\sqrt n\log n+2.
\]
\end{corollary}

\begin{proof}
The eigenvalues of $\mathcal L$ on $C_n$ are $4\sin^2(\pi a/n)$ with orthonormal eigenvectors whose entries have modulus $n^{-1/2}$. Since $\sin(\pi a/n)\ge2\min\{a,n-a\}/n$,
\[
 (\mathcal L^+)_{vv}=\frac1n\sum_{a=1}^{n-1}\frac1{4\sin^2(\pi a/n)}
 \le\frac1n\cdot2\cdot\frac{n^2}{16}\sum_{r\ge1}\frac1{r^2}\le\frac n4.
\]
Apply \cref{thm:smoothed} with $d=2$, $\zeta=10$, and $\beta=22$, so that $\theta=242\log n$. The bound \eqref{eq:smoothed-bound} is at most
\[
 44\log n\bigl(2\sqrt{5.5n}+\tfrac43\bigr)+2\le207\sqrt n\log n+59\log n+2\le300\sqrt n\log n
\]
for $n\ge3$. For physical time, the number $N_t$ of events is Poisson with mean $\mu=nt\le n^5$, and $\Gap(X_t)\le N_t$. Hence
\[
 \E\Gap(X_t)\le300\sqrt n\log n+\E\bigl[N_t\one_{\{N_t>n^{10}\}}\bigr]
 \le300\sqrt n\log n+\frac{\mu+\mu^2}{n^{10}}\le300\sqrt n\log n+2.
\]
\end{proof}

\begin{remark}[Gaussian free field scale and the square torus]\label{rem:smoothed-gff}
Let $\eta$ be the centered Gaussian vector with covariance $d\mathcal L^+$. Then $\sum_v\eta_v=0$ almost surely, so $\max_v\eta_v-\min_v\eta_v\ge|\eta_w|$ for every $w$. With $\sigma_w^2=d(\mathcal L^+)_{ww}$ this gives $\E(\max\eta-\min\eta)\ge\sqrt{2/\pi}\,\max_w\sigma_w$, and \eqref{eq:smoothed-bound} implies
\[
 \E\Gap(X^{[k]})\le 2\beta\log N\Bigl(2\sqrt{\pi(4d+3)/2}\ \E(\max\eta-\min\eta)+\frac43\Bigr)+2.
\]
The additive constants are absorbed into this comparison: since the nonzero Laplacian eigenvalues are at most $2d$, we have $d\mathsf R_G\ge(1-1/N)/2\ge1/3$. Thus, up to any fixed polynomial horizon, the smoothed threshold rule is within a factor $O_\zeta(\sqrt d\log N)$ of the expected range of this normalized Gaussian free field.

On the square torus $C_L\square C_L$, a Fourier computation gives $\mathsf R_G\le1+\log L$. Indeed, $4\sin^2(\pi a/L)\ge16a'^2/L^2$ with $a'=\min\{a,L-a\}$, each pair $(a',b')$ arises from at most four frequencies, $a'^2+b'^2\ge(a'+b')^2/2$, and at most $2s$ pairs have $a'+b'=s\ge1$; hence
\[
 (\mathcal L^+)_{vv}\le\frac1{L^2}\cdot4\cdot\frac{L^2}{16}\sum_{s=1}^{L}2s\cdot\frac2{s^2}\le1+\log L.
\]
Consequently the smoothed threshold rule has expected gap $O((\log L)^{3/2})$ up to any fixed polynomial horizon, while, for $L\ge32$, \cref{prop:universal-log} shows that every strategy has gap at least $(\log L)/32$ with probability at least $1/2$ after every $k\ge(L^2/8)\log L$ allocations.
\end{remark}

\section{Open problems}\label{sec:open}

The lower bound on the cycle has the expected saturated scale, and \cref{cor:smoothed-cycle} shows that it is attained within the endpoint-local monotone class up to a logarithmic factor. For greedy graphical allocation itself the corresponding upper bound remains open. The best known equilibrium upper bound is $O(n)$ \cite{OleskerTaylorSauerwaldZanetti2026}, so the stationary gap of greedy is presently known only between order $\sqrt n$ and order $n$.

\begin{problem}[Greedy cycle upper bound]
Does greedy allocation on $C_n$, started from the flat profile, satisfy
\[
 \sup_{t\ge0}\E\Gap(X_t)=O(\sqrt n)?
\]
Does every invariant law of the normalized process satisfy the corresponding stationary bound? A stationary estimate alone does not imply the uniform transient estimate without an additional convergence argument.
\end{problem}

On the square torus $C_L\square C_L$, every strategy has gap of order at least $\log L$ (\cref{prop:universal-log}), the maximum of the two-dimensional discrete Gaussian free field is of the same order \cite{BramsonDingZeitouni2016}, and the smoothed threshold rule has expected gap $O((\log L)^{3/2})$ up to polynomial horizons (\cref{rem:smoothed-gff}). For the different dynamic-averaging process, Berenbrink, Hintze, Hosseinpour, Kaaser, and Rau obtained an $O((\log N)^{3/2})$ bound on the two-dimensional torus using effective-resistance methods \cite{BerenbrinkEtAl2023}; Kraizberg subsequently sharpened this to the optimal $O(\log N)$ bound using pairwise concentration governed by effective resistance \cite{Kraizberg2026}.

\begin{problem}[Square torus and the Gaussian free field scale]
Does greedy allocation, or some endpoint-local monotone rule, have expected gap $O(\log L)$ on $C_L\square C_L$? More generally, on bounded-degree regular graphs, is the expected range of the Gaussian free field with covariance $d\mathcal L^+$ a lower bound, up to polylogarithmic factors, for the gap of every endpoint-local monotone rule at diffusive times? The transport-volume inequality controls one averaged family of two-point contrasts; a positive answer would require joint fluctuation information for many separated regions.
\end{problem}

Finally, a rule that reads loads in a bounded neighborhood need not preserve one unit discrepancy: an off-edge perturbation may change a remote endpoint decision and create a signed discrepancy with several nonzero coordinates. The obstruction is therefore not captured by a single tag.

\begin{problem}[Bounded-radius information]
Find a replacement for the unit-discrepancy tag for rules whose decision on an edge may inspect loads within graph distance $r_0$, where $r_0$ is fixed. Does every such monotone rule still incur a polynomial gap on the cycle, and what exponent is forced by the propagation of its signed response?
\end{problem}

\begingroup
\hbadness=10000

\endgroup

\end{document}